\documentclass[12pt,reqno]{amsart}
\usepackage{amsmath,amsthm,amsfonts,amssymb}
\usepackage[french,english]{babel}
\usepackage{tikz}
\usepackage{psfrag,pslatex}
\usepackage{epsfig}

\numberwithin{equation}{section}

\newcommand{\dist}{\operatorname{dist}}

\newcommand{\de}{\em}

\newcommand{\diam}{\operatorname{diam}}

 \newcommand{\eps} {\varepsilon}

\def \T {{\mathbb T}}

\newtheorem{theorem}{Theorem}[section]

\newtheorem{proposition}[theorem]{Proposition}

\newtheorem{lemma}[theorem]{Lemma}
\newtheorem{conjecture}[theorem]{Conjecture}

\theoremstyle{remark}

\begin{document}

\renewcommand{\subjclassname}{\textup{2000} Mathematics Subject Classification}


\setcounter{tocdepth}{2}



\renewcommand{\subjclassname}{\textup{2000} Mathematics Subject Classification}

\date{\today}

\title{Center-unstable foliations do not have compact leaves}

\begin{abstract}
For a partially hyperbolic diffeomorphism on a 3-manifold,  we show that any invariant foliation tangent to the center-unstable (or center-stable) bundle has no compact leaves.
\end{abstract}

\thanks{}

\author{F. Rodriguez Hertz}
\address{Department of Mathematics\\ The Pennsylvania State University,
University Park,
State College, PA 16802 .}
\email{hertz@math.psu.edu}

\author{J. Rodriguez Hertz}
\address{IMERL-Facultad de Ingenier\'\i a\\ Universidad de la
Rep\'ublica\\ CC 30 Montevideo, Uruguay.}
\email{jana@fing.edu.uy}\urladdr{http://www.fing.edu.uy/$\sim$jana}

\author{R. Ures}
\address{IMERL-Facultad de Ingenier\'\i a\\ Universidad de la
Rep\'ublica\\ CC 30 Montevideo, Uruguay.} \email{ures@fing.edu.uy}
\urladdr{http://www.fing.edu.uy/$\sim$ures}

\maketitle
\section{Introduction}

This article deals with the integrability of invariant distributions arising in certain dynamical systems. 
The integrability of tangent distributions of $k-$planes is an important problem
that has been studied for more than a century. Under certain hypotheses, there have been 
quite satisfactory answers to the unique integrability problem. For instance, Picard's theorem in the one-dimensional case, and Frobenius' theory in higher dimensions. However, both results involve regularity of the distributions. \par
The problem is that, in general, distributions arising from a dynamical system are only H\"older-continuous, even if the system is smooth. Therefore, other elements of analysis are required in order to establish necessary or sufficient conditions for integrability. \par

Consider, for instance, the case of hyperbolic systems. Namely, diffeomorphims $f:M\to M$ for which there exists an invariant splitting of the tangent bundle into two invariant distributions, called {\em stable} and {\em unstable}: $TM=E^{s}\oplus E^{u}$ such that, for some Riemannian metric $\|.\|$ and all unit vectors $v_{s}\in E^{s}$ and $v_{u}\in E^{u}$ we have:
$$\|Df(x)v_{s}\|<1<\|Df(x)v_{u}\|.$$
In this case, both distributions, stable and unstable, are uniquely integrable. This can be achieved by applying essentially either one of the following two methods: Hadamard's method \cite{hadamard}, which consists in seeing each invariant integral manifold as a fixed point of a contracting operator acting in an appropriate space of functions, or Perron's method \cite{perron}, which consists in applying the Implicit Function Theorem to an operator acting in an appropriate function space.\par
Here we shall focus on the integrability of some distributions arising from a more general kind of systems, the {\em partially hyperbolic} ones. These systems involve not only a contracting and an expanding bundle as the ones mentioned above, but also a {\em center bundle}, which has an intermediate behavior. Namely; a diffeomorphism $f:M\to M$ is {\em partially hyperbolic} if the tangent bundle of $M$ admits an invariant splitting into 3 bundles, called respectively, {\em stable}, {\em center} and {\em unstable} bundles, $TM=E^{s}\oplus E^{c}\oplus E^{u}$, such that, for some Riemannian metric $\|.\|$, and all unit vectors $v_{s}\in E^{s}$, $v_{c}\in E^{c}$ and $v_{u}\in E^{u}$ we have 
$$\|Df(x)v_{s}\|<\|Df(x)v_{c}\|<\|Df(x)v_{u}\|\quad\text{and}\quad \|Df(x)v_{s}\|<1<\|Df(x)v_{u}\|$$
Hadamard's and Perron's methods can be used to show that also in this setting, the stable and unstable bundles are uniquely integrable, see for instance \cite{bp,hps}.
However, it is a more delicate matter to determine wether $E^{c}$ or even $E^{cs}=E^{s}\oplus E^{c}$ or $E^{cu}=E^{c}\oplus E^{u}$ are integrable. \par
We will say that $f$ is {\em $cs$-dynamically coherent} if there exists an $f$-invariant foliation tangent to $E^{cs}$. The $cu$-dynamical coherence is defined analogously. A diffeomorphism $f$ is said to be {\em dynamically coherent} if it is both $cs$- and $cu$-dynamically coherent. This implies in particular that there exists an $f$-invariant foliation tangent to $E^{c}$.\par
Partially hyperbolic diffeomorphisms are not dynamically coherent in general. Indeed, as it was observed by A. Wilkinson in \cite{wilk}, a non-dynamically coherent example is given by an algebraic Anosov diffeomorphim in a six dimensional manifold, that was presented in the well-known survey by S. Smale \cite{sm} and is attributed to A. Borel. In this example, the sub-bundles are $C^{\infty}$, the center bundle is 4-dimensional and it cannot be integrable since the Frobenius condition is not satisfied. \newline\par

Is the lack of Frobenius condition the only reason for non-integrability of the center bundle? What about the case of one-dimensional $E^{c}$, where Frobenius condition is always trivially satisfied? The question of whether a partially hyperbolic diffeomorphism existed for a one-dimensional non-integrable center bundle remained open since the 70's. In \cite{HHU.nondinco} the authors answered the question negatively. They constructed examples of non-dynamically coherent
partially hyperbolic diffeomorphisms on  $\mathbb{T}^3$. In \cite{HHU.nondinco}, it is also given (using the same methods) an example of a dynamically coherent diffeomorphism with non-locally uniquely integrable center foliation. The existence of such examples
contrasts with the result obtained by M. Brin, D. Burago and S.
Ivanov \cite{brin_burago_ivanov2} proving that diffeomorphisms on
$\T^3$ satisfying a more restrictive definition of partial hyperbolicity (the {\em absolute partial hyperbolicity}) are dynamically coherent. \newline\par

A natural question is then: are there necessary or sufficient conditions for dynamical coherence? In this paper we provide some necessary conditions:

\begin{theorem}\label{uniqueint} Let $M$ be a closed 3-dimensional manifold and $f:M\rightarrow M$ be a $cu$-dynamically coherent partially hyperbolic diffeomorphism. Then, the center unstable foliation has no compact leaves.
\end{theorem}
Let us note that any compact leaf tangent to $E^{cu}$ must be a torus, by Poincar\'e-Hopf, due to the fact that is is foliated by lines. We shall call $cu$-torus any 2-torus tangent to $E^{cu}$.\par
As a matter of fact, the non-dynamically coherent example \cite{HHU.nondinco} was inspired in this result. Since any center-unstable foliation cannot have a torus leaf, we constructed a plane field that was uniquely integrable outside a $cu$-torus in such a way that if a foliation existed, it should contain the torus as a leaf. The fact that this is not possible provided the first example of a non dynamically coherent partially hyperbolic diffeomorphism with one-dimensional center bundle. \par
Observe that Theorem \ref{uniqueint} does not prevent the existence of tori, even invariant, tangent to $E^{cu}$. Theorem \ref{uniqueint} asserts the impossibility of the existence of such tori {\em as part of an invariant foliation} tangent to $E^{cu}$. Indeed, in \cite{HHU.nondinco} we also provide an example of a dynamically coherent partially hyperbolic diffeomorphism of $\T^{3}$ with an invariant $cu$-torus. \par
Theorem \ref{uniqueint} then states that if there is a center-unstable invariant foliation, no leaf can be a torus. We conjecture the converse is also true:

\begin{conjecture}\label{conj}
If a  partially hyperbolic diffeomorphism $f:M^3\rightarrow M^3$ is not dynamically coherent, then it admits either a $cu$- or an $sc$-torus.
\end{conjecture}

In Section \ref{section.cu.tori} we prove that, in fact, not every manifold can contain a $cu$- or an $sc$-torus:
\begin{theorem}\label{invariant.tori} Let $f:M\to M$ be a partially hyperbolic diffeomorphism of an orientable 3-manifold. If there exists a torus tangent to either $E^{s}\oplus E^{u}$, $E^{c}\oplus E^{u}$ or $E^{s}\oplus E^{c}$, then the manifold $M$ is either:
\begin{enumerate}
 \item the 3-torus $\T^{3}$
 \item the mapping torus of $-id:\T^{2}\to\T^{2}$
 \item the mapping torus of a hyperbolic automorphism $A:\T^{2}\to \T^{2}$
\end{enumerate}
\end{theorem}

This follows essentially from Proposition \ref{cu.tori.anosov}, where we prove that the dynamics on any invariant torus tangent to $E^{cu}$, or $E^{sc}$ is isotopic to the one generated by a hyperbolic linear automorphism. When a manifold $M$ admits an embedded 2-torus and a global diffeomorphism $g:M\to M$ such that $g(T)=T$ and $g|_{T}$ is isotopic to an Anosov diffeomorphism, then we say that $M$ admits an {\em Anosov torus}.\par
In \cite{anosov_tori}, we proved that the only irreducible manifolds admitting an Anosov torus are the ones listed in Theorem \ref{invariant.tori}. But on the other hand, if $M$ admits a partially hyperbolic diffeomorphism, then $M$ is an irreducible manifold. See Section \ref{section.cu.tori} for more details.\par

As a consequence of Theorem \ref{invariant.tori}, if Conjecture \ref{conj} were true, then the only manifolds supporting non-dynamically coherent diffeomorphisms would be the ones listed in Theorem \ref{invariant.tori}. In Proposition \ref{cu.tori.periodic} we show that the existence of a $cu$-torus implies the existence of a {\em periodic} $cu$-torus, which is attracting. An analogous statement holds for $sc$-tori. Therefore, were Conjecture \ref{conj} true, all partially hyperbolic diffeomorphisms $f$ with $\Omega (f)=M$ would be dynamically coherent. Hammerlindl and Potrie have proven that Conjecture \ref{conj} is true for manifolds with solvable fundamental group \cite{hp}; namely, for the manifolds that are finitely covered by the ones listed in Theorem \ref{invariant.tori}. This is the greatest advance in Conjecture \ref{conj} so far. Theorem \ref{uniqueint} is used in their proof. \newline \par

A foliation is {\de taut } if there is an embedded circle that transversely intersects each one of its leaves. Taut foliations play an important role in the description of 3-dimensional manifolds. In Section \ref{taut} we show:
\begin{theorem} \label{theorem.taut} Let $E^s$ be the strong stable bundle of a partially hyperbolic diffeomorphism of a closed orientable 3-dimensional manifold $M$. If $\mathcal{F}$ is a foliation transverse to $E^s$ that is not taut, then there exists a periodic $cu$-torus. In particular, $M$ admits an Anosov torus.
\end{theorem}

A foliation ${\mathcal F}$ like the one mentioned in Theorem \ref{theorem.taut} always exists, due to Burago-Ivanov  \cite{burago_ivanov}, see more details in Theorem \ref{burago.ivanov}. As a consequence of Theorem \ref{theorem.taut}, all manifolds supporting partially hyperbolic diffeomorphisms are finitely covered by manifolds supporting taut foliations. Perhaps the theory of taut foliations could give some enlightening to the description of partially hyperbolic systems (see, for instance, \cite{calegari}).

\section{Dynamics on $cu$- and $sc$-tori}\label{section.cu.tori}
In this section, we shall prove Theorem \ref{invariant.tori}, which will follow from certain dynamical properties of $cu$- and $sc$-tori. \par
As we said in the Introduction, a manifold $M$ {\em admits an Anosov torus} if there exists a diffeomorphism $g:M\to M$ and an embedded $g$-invariant torus $T$ such that $(g|_{T})_*:\pi_1(\mathbb{T}^2)\rightarrow \pi_1(\mathbb{T}^2)$ is hyperbolic. Admitting an Anosov torus is a global property. In \cite{anosov_tori}, we prove that very few 3-manifolds have such a property.

\begin{theorem}\cite{anosov_tori} Let $M$ be an irreducible orientable 3-manifold admitting an Anosov torus, then
$M$ is either:
\begin{enumerate}
  \item the 3-torus $\T^{3}$
 \item the mapping torus of $-id:\T^{2}\to\T^{2}$
 \item the mapping torus of a hyperbolic automorphism $A:\T^{2}\to \T^{2}$
\end{enumerate} 
\end{theorem}

A 3-manifold admitting a partially hyperbolic diffeomorphism is always irreducible, see \cite[Lemma 6.3]{survey}. Theorem \ref{invariant.tori} then follows from the following propositions:

\begin{proposition}\label{cu.tori.periodic}
 The existence of a $cu$-torus implies the existence of a periodic $cu$-torus.
\end{proposition}

\begin{proposition}\label{cu.tori.anosov}
The dynamics on an invariant $cu$-torus is isotopic to hyperbolic. 
\end{proposition}

\begin{proposition}\label{su.torus.anosov}
 A manifold admitting an $su$-torus, admits an Anosov torus. 
\end{proposition}

Proposition \ref{cu.tori.anosov} is a direct corollary of Lemma \ref{hyperbolic} below. See also Proposition 2.1 of \cite{brin_burago_ivanov} for a similar result and proof.

\begin{lemma}\label{hyperbolic} Let $\mathcal{W}$ be a foliation of $\,\mathbb{T}^2$ with continuous tangent bundle $T\mathcal{W}$ and invariant by a diffeomorphism $g$. Suppose, in addition, that $||dg|_{T\mathcal{W}}||>1$. Then, $g_*:\pi_1(\mathbb{T}^2)\rightarrow \pi_1(\mathbb{T}^2)$ is hyperbolic.
\end{lemma}

\begin{proof}
By taking $g^2$ if necessary we  can suppose that $g$ preserves the orientation of $T\mathcal{W}$. Since $g$ preserves a foliation without compact leaves,  $g_*:\mathbb{Z}^2\rightarrow \mathbb{Z}^2$ (we identify $\pi_1(\mathbb{T}^2)$ with $\mathbb{Z}^2$) has an eingenspace of irrational slope. This implies that either $g_*$ is hyperbolic or $g_*=Id$. In the second case $g$ has a lift $\tilde g:\mathbb{R}^2\rightarrow \mathbb{R}^2$ such that $\tilde g=Id+\alpha$ where $\alpha$ is a periodic, and in particular bounded, function. As a consequence we obtain that there exists a constant $K>0$ such that given any subset of $\mathbb{R}^2$, $X$, $\diam(g^n(X))\leq \diam(X)+nK$. Let $\gamma$ be an arc contained in a leaf of $\mathcal{W}$. Then, the length of $\gamma$ grows exponentially while its diameter grows at most linearly. This implies that given a small $\varepsilon>0$ there exists an iterate of $\gamma$ that contains a curve of length arbitrarily large and with end points at distance less than $\varepsilon$. Using Poincar\'{e}-Bendixon we obtain a compact leaf. This  is a contradiction and then, $g_*$ is hyperbolic.
\end{proof}

\begin{proof}[Proof of Proposition \ref{cu.tori.periodic}]  Let $T$ be a $cu$-torus, and consider the sequence $f^{-n}(T)$. Since the family of all compact subsets of $M$, considered with the Hausdorff metric $d_{H}$, is compact, there is a subsequence $f^{-n_{k}}(T)$ converging to a compact set $K\subset M$. Therefore, for each $\eps>0$ there are arbitrarily large $N>>L>0$ such that $d_{H}(f^{-N}(T), f^{-L}(T))<\eps$. \par
 Since $T$ is transverse to the stable foliation, the union of all local stable leaves of $T$ forms a small tubular neighborhood of $T$, $U(T)$. Since stable leaves grow exponentially under $f^{-1}$, if $N>>L$ as above are large enough, then $f^{-L}(\overline{U(T)})\subset f^{-N}(U(T))$. This implies that 
 $f^{N-L}(\overline{U(T)})\subset U(T)$. 

The set $T_{0}=\cap_{k=0}^{\infty}f^{k(N-L)}(U(T))$ is a periodic $cu$-torus. Indeed, it is easy to see that $T_{0}$ is periodic and homeomorphic to a torus. On the other hand, for each point $x$ of $T_{0}$, its tangent space is limit of the tangent spaces of points $x_{k}$ in $f^{k(N-L)}(T)$, which are $cu$-tori. Hence $T_{x}T_{0}=E^{c}_{x}\oplus E^{u}_{x}$.
\end{proof}

\begin{proof}[Proof of Proposition \ref{su.torus.anosov}] Assume $f$ admits an $su$-torus, and consider the lamination $\Lambda$ of all $su$-tori of $f$. This is a compact lamination \cite{Hae62}. Therefore, there is a recurrent leaf., that is, there is a torus $T$ and an iterate $n$, such that $d_{C^{1}}(f^{n}(T),T)<\eps$ for small $\eps$. There exists a dipheotopy $i_{t}$ on $M$, taking $f^{n}(T)$ into $T$. Then $\phi=f^{n}\circ i_{1}$ fixes $T$ and $\phi|T$ is isotopic to an Anosov diffeomorphism.  
\end{proof}

\section{Weak foliations of partially hyperbolic diffeomorphisms}\label{taut}

In this section we prove Theorem \ref{theorem.taut}. For any partially hyperbolic diffeomorphism of a 3-manifold such that the invariant bundles are orientable,  Burago and Ivanov \cite{burago_ivanov} have proved  that there are (not necessarily invariant) foliations that ``almost" integrate $E^{c\sigma}=E^c\oplus E^\sigma, \,\,\, \sigma=s,\,u$.

\begin{theorem}[Key Lemma 2.2, \cite{burago_ivanov}]\label{burago.ivanov} Let $f$ be a partially hyperbolic diffeomorphism of a closed 3-manifold  and  let $E^*$ be orientable for $*=s,\, c, \, u$.
Then, for every $\varepsilon> 0$ there is a foliation $\mathcal{F}^{c\sigma}_\varepsilon$ such that  $T\mathcal{F}^{c\sigma}_\varepsilon$
 is a continuous bundle and the angles between $T\mathcal{F}^{c\sigma}_\varepsilon$ and $E^{c\sigma}$ are no greater than
$\varepsilon$, $\sigma=s,\, u$.
\end{theorem}

In this section we prove that these foliations, if the manifold is different from the ones listed in Theorem \ref{invariant.tori}, are taut. Recall that a codimension one foliation is {\em taut} if there exists an embedded $\mathbb{S}^1$ that intersects transversely, and nontrivially, every leaf of the foliation (see \cite{calegari}).\newline\par
Let ${\mathcal F}$ be a codimension-one foliation. A  {\em dead-end component} is an open submanifold
$N \varsubsetneq M$ which is a union of leaves of $\mathcal{F}$, such that there is
no properly immersed line transverse to $\mathcal{F}$. That is, there is no $\alpha:[a,b]\to M$ transverse to ${\mathcal F}$ such that 
$\alpha(a,b)\subset N$, $\alpha(a),\alpha(b)\in\partial N$. 

A  {\de Reeb
 component} is a solid torus whose interior is foliated by planes transverse
 to the core of the solid torus, such that each leaf limits on the
 boundary torus, which is also a leaf. Observe that the interior of a Reeb component is a dead-end component.
 Other examples of dead-end components are obtained by taking a Reeb foliation of the annulus, multiplying by the interval and gluing the two boundary annulus using a rotation that preserves the Reeb foliation.
Observe that in both cases the boundary of the dead-end component consists of tori. This is a general fact that is stated in the following lemma.

\begin{lemma}[Lemma 4.28, \cite{calegari}]\label{nodead} Let $M$ be a 3-dimensional orientable closed manifold. A foliation $\mathcal{F}$ of $M$ is taut if and only if it contains no dead-end components. If
$N$ is a dead-end component, then the restriction of $\mathcal{F}$ to $\overline N$ is transversely orientable and $\overline N\setminus N$ consists of a union of tori leaves of F. Moreover, boundary leaves of $N$ cannot be joined by an arc in $N$ transverse to $\mathcal{F}$.
\end{lemma}

Observe that the last assertion implies that the boundary leaves of any dead-end component of a foliation have half-neighborhoods in $N$ with size uniformly bounded by below.\par

It is an obvious corollary of Lemma \ref{nodead} that foliations without compact leaves are taut. For instance, the weak stable and weak unstable foliations of Anosov flows of 3-dimensional manifolds are taut.

\begin{proof}[Proof of Theorem \ref{theorem.taut}]
Partial hyperbolicity implies that the forward iterates of $T\mathcal{F}$ converge to $E^{cu}$. Let $N$ be a dead-end of $\mathcal{F}$ and let $T$ be a boundary component of $N$. Hence $T$ is a torus, transverse to $E^{s}$. All iterates $f^{-n}(T)$ are tori transverse to $E^{s}$. The proof follows now exactly as in Proposition \ref{cu.tori.periodic}.

\end{proof}

\section{Proof of Theorem \ref{uniqueint}}\label{s_teo2}

Let us suppose that there exists an invariant foliation $\mathcal{F}^{cu}$ tangent to $E^{cu}$, and that $\mathcal{F}^{cu}$ has a compact leaf, which must be a torus. (Then $M$ must be one of the manifolds listed in Theorem \ref{invariant.tori}).
By Proposition \ref{cu.tori.periodic}, there exists a periodic $cu$-torus $T$. By taking an iterate, we can assume that is fixed. The dynamics on $T$ is isotopic to hyperbolic, due to Proposition \ref{cu.tori.anosov}. In \cite{anosov_tori} it is shown that cutting $M$ along $T$ we obtain a manifold with boundary, that is diffeomorphic to $\mathbb{T}^2\times [0,1]$. Moreover, $f$ induces a diffeomorphism $g$ of  $\mathbb{T}^2\times [0,1]$ isotopic to $A\times id$ where $A$ is a hyperbolic automorphism of $\mathbb{T}^2$ and $id$ is the identity map of the interval $[0,1]$. Then, \cite{fr} implies that there exists a semiconjugacy $h:\mathbb{T}^2\times [0,1] \rightarrow \mathbb{T}^2$, between $g $ and $A$, homotopic to the projection $p:\mathbb{T}^2\times [0,1] \rightarrow \mathbb{T}^2$. Observe also that $h(\mathbb{T}^2\times \{0\})=\mathbb{T}^2$.\newline\par

The torus $\mathbb{T}^2\times \{0\}$ is foliated by a foliation $\mathcal{S}^u$ by lines that are integral curves of the strong unstable foliation. Call $\tilde{\mathcal{S}}^u$ the lift of $\mathcal{S}^u$ to $\mathbb{R}^2$, the universal cover of $\mathbb{T}^2\times \{0\}$. It is not difficult to see that if $x, \,\, y$ are in the same leaf $S^u $ of $\tilde{\mathcal{S}}^u$ the fact that $\dist_u(x,y)$ goes to infinity implies that $dist(x,y)$ goes to infinity ($\dist_u(x,y)$ is the length of the arc of $S^u $ joining $x$ and $y$). Let $\bar h$ be a lift of $h|_{\mathbb{T}^2\times \{0\}}$ to $\mathbb{R}^2$. Since the diameters of the sets $\bar h^{-1}(y)$ are uniformly bounded and $f(\bar h^{-1}(y))=\bar h^{-1}(A y)$, we obtain that the map $\bar h$ is injective when restricted to strong unstable manifolds. Recall that the image of an unstable manifold by $h$ is an unstable manifold of $A$. \newline\par

Let us show that the image of a center curve by $h$ is contained in a stable manifold of $A$. For this, it is enough to show that the length of the forward iterates of the curve are bounded. Let $\gamma$ be a (small) center curve and let $\delta $ be such that $W^u_\delta(x)\cap W^u_\delta(y)=\emptyset$ for all $x\ne y \in \gamma$. Let $W^u_\delta(\gamma)=\bigcup\{W^u_\delta(x);\, x\in \gamma\}$. Since $f$ expands the unstable bundle we have that $W^u_\delta(f^n(\gamma))\subset f^n(W^u_\delta(\gamma))$.
But since the angle between the center bundle and the unstable bundle is bounded by below, we have that the area of $W^u_\delta(f^n(\gamma))$ goes to infinity as the length of $f^n(\gamma)$ goes to infinity which is a contradiction. Moreover, this implies that $E^c$ is uniquely integrable for $f|_T$. If this were not the case there would exist, in the universal cover,  two center curves $\gamma_1,\,\,\gamma_2$ beginning at the same point $x $ and cutting a nearby unstable manifold at two different points $y,\, z$. By forward-iterating this ``triangle" we would obtain that the distance between $f^n(y)$ and $f^n(z)$ goes to infinity. Then, either the length of $\gamma_1 $ or the length of $\gamma_2$ goes to infinity, which contradicts what we have proved before. Summarizing, $f|_T$ has two invariant foliation by lines, one tangent to the unstable bundle and the other one tangent to the center bundle; the semiconjugacy sends the unstable leaves to unstable lines of $A$ and the center leaves to stable lines of $A$. Again the distance between two points in the same center leaf in the universal cover of $T$ goes to infinity as the length of the center curve goes to infinity. It is not difficult to see that this implies that, for any $y\in \mathbb{T}^2$, $h^{-1}(y)\cap T$ is a connected arc contained in a center leaf.\newline\par

Let $p\in T$ be a periodic point. We may assume that $J_p=h^{-1}(h(p))\cap T$ is a very small arc. This is easy to obtain since there are infinitely many periodic points in different center curves. Let $U\subset \mathbb{T}^2\times I$ be a small neighborhood of $J_p$, and let $y\in h^{-1}(h(p))\cap U$. We will prove that $y\in W^{cs}_{loc}(p)$. On one hand, if $y\notin W^{cs}_{loc}(p)$ then $z=W^s_{loc}(y)\cap T$ is not in the local center manifold of $p$. On the other hand, $h(z)\in h(W^s_{loc}(y))\subset W^s_{loc}(h(p))$ is in the local stable manifold of $h(p)$ for $A$, which contradicts the fact that it is not in the local center manifold (in $T$) of $p$. \par

Let us focus on $W^{cs}_{loc}(p)$. The intersection of the center unstable foliation $\mathcal{F}^{cu} $ with $W^{cs}_{loc}(p)$ foliates $W^{cs}_{loc}(p)$ by center arcs.
By continuity, any of these center arcs, close enough to $T$, has a point of $h^{-1}(h(p))$. Certainly, the same is true for $\bar p$, a lift of $p$ to universal cover. We choose $x_1,\dots, x_N\in W^{cs}_{loc}(\bar p)$ points that are in $\bar h^{-1}(\bar h(\bar p))$ and in different center curves. Let $C>0$ be such that $\diam (\bar h^{-1}(y))<C$ for every $y\in T$ and $\varepsilon>0$ so small that if two points are at a distance less than  $\varepsilon$, then they are in a trivializing chart of $\mathcal{F}^{cu}$. Now, we choose $N$ in such a way that if $N$ points are contained in a set of diameter $C$  then, at least two of them are a distance less than $\varepsilon$. Since $\bar f^n(\bar h^{-1}(\bar h(\bar p)))=\bar h^{-1}(\bar h(\bar f^n(\bar p)))$ we have that $\diam \{\bar f^n(x_1),\dots ,\bar f^n (x_n)\}<C, \,\, \forall n\in \mathbb{Z}$. Then, there exists a subsequence $n_k\rightarrow -\infty$ and two different points $x_i$ and $x_j$ such that $\dist(\bar f^{n_k}(x_i),\bar f^{n_k}(x_j))<\varepsilon, \,\, \forall k>0$.  Now, take an arc $\alpha$ joining $x_i$ and $x_j$ that consists of two sub-arcs, $\alpha_1$ beginning at $x_i $ and tangent to the center bundle and $\alpha_2$ ending at $x_j$ and tangent to stable one (see Figure \ref{figure3}).

\begin{figure}
\begin{tikzpicture}
 \draw (-3,-3) -- (5,-3);
 \draw (-3,-2.75) -- (5,-2.75);
 \draw (-3,-2.5) -- (5,-2.5);
 \draw (-3,-2.25) -- (5,-2.25);
 \draw (-3,-2) -- (5,-2);
 \draw (-3,-1.75) -- (5,-1.75);
 \draw (-3,-1.5) -- (5,-1.5);
 \draw (-3,-1.25) -- (5,-1.25);
 \draw (-3,-1) -- (5,-1);
 \draw (-3,-0.75) -- (5,-0.75);
 \draw (-3,-0.5) -- (5,-0.5);
 \draw (-3,-0.25) -- (5,-0.25);
 \draw (-3,0) -- (5,0);
 \draw (-3,-3.25) -- (5,-3.25);
 \draw[line width=2pt] (-1,-3.25) -- (1,-3.25);
  \draw[line width=2pt] (0,-2.5) -- (4,-2.5)-- (4,-0.5);

 \fill (0,-2.5) circle (3pt);
 \fill (4,-0.5) circle (3pt);
 \node at (0,-2.80) {$x_i$};
  \node at (2,-2.15) {$\alpha_1\subset W^c(x_i)$};
\node at (4,-0.2) {$x_j$};

\node at (5.5,-1.2) {$\alpha_2\subset W^s(x_j)$};
 \node at (0,-3.75) {$J_{\bar p}$};
\end{tikzpicture}
\caption{\label{figure3}
 $W^{cs}(\bar p)$}
\end{figure}
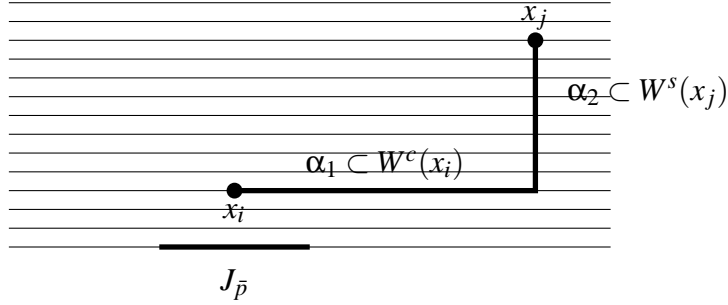

For $k$ large enough we obtain that $\bar f^{n_k}(\alpha_2)$ is a very long stable curve, $\bar f^{n_k}(\alpha_1)$ is contained in a leaf of $\tilde {\mathcal{F}}^{cu}$ and the extremes of $\bar f^{n_k}(\alpha)$ are a distance less than $\varepsilon$ (see Figure \ref{figure_reeb}).

 \begin{figure}[h]
 \vspace{-7cm}
\begin{tikzpicture}


\draw (1.3,-1) .. controls +(-110:7cm) and +(100:10cm) .. (3,0);
\node at (0,1) {$\bar f^n(\alpha_2)$};
\draw (3,0) -- (1.7,-0.1);
\node at (2.2,0.3) {$\bar f^n(\alpha_1)$};
\fill (1.3,-1) circle (2pt);
\node at (2,-1) {$\bar f^n(x_2)$};
\fill (1.7,-0.1) circle (2pt);
\node at (1,-0.1) {$\bar f^n(x_1)$};
\draw[thick,dotted] (0.8,-0.7) -- (4.3,-0.7) -- (5.8,1.5) -- (2.3,1.5) -- cycle;
\node at (4.2,1) {$\bar F^{cu}(\bar f^n(x_1))$};
\end{tikzpicture}
\vspace{-4.5cm}\caption{\label{figure_reeb}
 $\bar f^n(\alpha)$}
\end{figure}
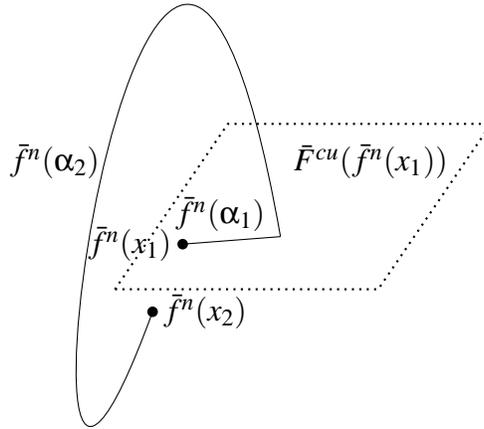

 Standard arguments of foliation theory imply that there is a closed curve transverse to $\tilde {\mathcal{F}}^{cu}$ which implies the existence of a Reeb component which is, as it is well known (see for instance \cite{burago_ivanov,rru_nilman}), impossible. This finishes the proof of Theorem \ref{uniqueint}.


\end{document}